\documentclass[11pt,letterpaper]{article}

\usepackage[utf8]{inputenc}
\usepackage[T1]{fontenc}
\usepackage{mathtools}
\usepackage{amssymb}
\usepackage{amsthm}
\usepackage{newtxtext}
\usepackage{newtxmath}
\usepackage{booktabs}
\usepackage{enumitem}
\usepackage[numbers,sort&compress]{natbib}
\usepackage[margin=1in]{geometry}
\usepackage{microtype}
\usepackage[hidelinks]{hyperref}

\numberwithin{equation}{section}
\allowdisplaybreaks[2]

\newcommand{\R}{\mathbb{R}}
\newcommand{\Q}{\mathbb{Q}}
\newcommand{\proofstep}[1]{%
  \par\medskip
  \noindent\textbf{#1}%
  \par\smallskip
}

\theoremstyle{plain}
\newtheorem{theorem}{Theorem}[section]
\newtheorem{lemma}[theorem]{Lemma}
\newtheorem{corollary}[theorem]{Corollary}

\theoremstyle{remark}
\newtheorem{remark}[theorem]{Remark}

\begin{document}

\title{An Inverse Theorem for Partially Symmetric Two-Dimensional
Semiclassical Schr\"odinger Operators}

\author{Kuo Wang\\
  \small School of Mathematical Sciences, Nankai University, Tianjin, China\\
  \small \href{mailto:1120240038@mail.nankai.edu.cn}{1120240038@mail.nankai.edu.cn}}

\date{\today}

\maketitle

\begin{abstract}
We study a two-dimensional semiclassical Schr\"odinger operator whose potential 
admits one reflection symmetry. Under a rational independence assumption on the 
harmonic frequencies and a nonvanishing condition \(a_{30}\ne0\), we show that the 
first two layers of the quantum Birkhoff normal form, together with the sign of 
\(a_{30}\) and the transverse data \(\{a_{1,2k}\}_{k\ge1}\), determine the full 
Taylor series of the potential at the bottom of the well. The proof is constructive: 
the \(\hbar^2\)-layer gives a triangular recursion for odd-degree terms, while the 
classical layer recovers even-degree terms from their resonant projections.
\end{abstract}

\section{Introduction}

We consider the two-dimensional semiclassical Schr\"odinger operator
\begin{equation}
\label{eq:Schrodinger}
\hat H_\hbar=-\frac{\hbar^2}{2}\Delta+V(q),
\qquad q\in\R^2.
\end{equation}
Assume that the real-valued potential $V\in C^\infty(\R^2;\R)$ has a unique
nondegenerate global minimum at the origin, with $V(0)=0$, and that
$V^{-1}([0,\epsilon])$ is compact for some sufficiently small $\epsilon>0$.
Then, for every sufficiently small $\delta\in(0,\epsilon)$ and sufficiently
small $\hbar$, the spectrum of $\hat H_\hbar$ in $[0,\delta]$ consists of
finitely many eigenvalues. We denote them by
$\{E_j(\hbar)\}_{j=1}^{N_\hbar(\delta)}$. The Weyl law gives
\begin{equation}
\label{eq:weyl}
N_\hbar(\delta)
:=\#\{j:0\le E_j(\hbar)\le\delta\}
=(2\pi\hbar)^{-2}\left(
\operatorname{Vol}\left\{0\le\frac12\lVert p\rVert^2+V(q)\le\delta\right\}
+o(1)\right)
\end{equation}
as $\hbar\downarrow0$; see \cite{DS99}.

We study the following local semiclassical inverse problem: to what extent do
the low-lying eigenvalues determine the Taylor expansion of $V$ at the origin
as $\hbar\downarrow0$?

This problem can be approached through the quantum Birkhoff normal form, which
is a spectral invariant. If the harmonic frequencies satisfy a nonresonance
condition, a formal conjugation brings \eqref{eq:Schrodinger} into a normal form
whose Weyl symbol depends only on
$\Omega_j=x_j^2+\xi_j^2$ and on even powers of $\hbar$; see
\cite{ISJ02,VN06}:
\begin{equation}
\label{eq:QBNF}
B\equiv H_2+\sum_{2r+k+\ell\ge2}b_{r,k,\ell}\,
\hbar^{2r}\Omega_1^k\Omega_2^\ell.
\end{equation}
When the harmonic frequencies are rationally independent, the low-lying
spectral data in a sufficiently small fixed energy interval determine the QBNF;
see \cite{GU07,GPU07,CdV09}. Thus, once the quantum Birkhoff normal form is
known, the inverse spectral problem becomes an algebraic one: can the Taylor
coefficients of the potential be recovered from the coefficients of the normal
form?

In the existing literature, this reconstruction depends strongly on symmetry.
If the potential is even in each coordinate, the classical Birkhoff normal form,
that is, the $\hbar^0$ part of the quantum normal form, already determines its
Taylor series \cite{GU07}. Without additional global hypotheses, even the complete bottom of the well semiclassical spectral asymptotics do not in general determine the potential globally; see the examples of Guillemin--Hezari \cite{GH12} and West
\cite{West23}.

Colin de Verdi\`ere and Guillemin \cite{CdVG11} showed in dimension one that the
quantum part of the normal form, beginning with the $\hbar^2$ terms, can
compensate for the lack of evenness. In higher dimensions, known results
usually impose additional structural restrictions. For example, after a
relabeling of coordinates, Hezari \cite{Hez09} obtained reconstruction for
potentials of the form
\[
V(x)=f(x_1^2,\dots,x_n^2)+x_n^3g(x_1^2,\dots,x_n^2).
\]
In the two-dimensional result of Guillemin and Uribe \cite{GU11}, coordinatewise
evenness is weakened to the reflection symmetry
$V(x_1,x_2)=V(x_1,-x_2)$ together with the additional axial condition
$V(x_1,0)=V(-x_1,0)$ and the genericity assumption
$\partial_{x_1}\partial_{x_2}^2V(0)\ne0$.

We investigate an intermediate two-dimensional setting with only the reflection
symmetry
\begin{equation}
\label{eq:symmetry}
V(q_1,q_2)=V(q_1,-q_2).
\end{equation}
Near the origin, write
\begin{equation}
\label{eq:V_taylor}
V(q_1,q_2)=\frac12\bigl(v_1^2q_1^2+v_2^2q_2^2\bigr)
 +\sum_{j+2k\ge3}c_{j,2k}\,q_1^jq_2^{2k},
\end{equation}
where $v_1,v_2>0$. Introduce the linear symplectic coordinates
\begin{equation}
\label{eq:linear_normalization}
x_i=v_i^{1/2}q_i,
\qquad
\xi_i=v_i^{-1/2}p_i,
\qquad i=1,2.
\end{equation}
After the corresponding linear conjugation, the Weyl symbol becomes
\begin{equation}
\label{eq:normalized_H}
H=H_2+U(x_1,x_2),
\qquad
H_2=\frac12\sum_{i=1}^2v_i(x_i^2+\xi_i^2),
\end{equation}
where
\begin{equation}
\label{eq:normalized_U}
U(x_1,x_2)=\sum_{j+2k\ge3}a_{j,2k}\,x_1^jx_2^{2k},
\qquad
a_{j,2k}=c_{j,2k}v_1^{-j/2}v_2^{-k}.
\end{equation}
We label $v_1$ as the frequency tangent to the reflection axis $q_2=0$ and
$v_2$ as the frequency in the normal direction, and keep this labeling fixed.
Recovering the Taylor coefficients of $U$ is equivalent to recovering those of
$V$.

\begin{theorem}
\label{thm:main}
For the normalized Hamiltonian \eqref{eq:normalized_H}--\eqref{eq:normalized_U},
assume that $v_1/v_2\notin\Q$ and $a_{30}\ne0$. Given
\begin{enumerate}[label=\textup{(\roman*)},leftmargin=2.8em,itemsep=0.2ex]
\item the sign of $a_{30}$;
\item the sequence of normalized coefficients $\{a_{1,2k}\}_{k\ge1}$,
\end{enumerate}
the first two layers of the quantum Birkhoff normal form, namely the
coefficients $b_{0,k,\ell}$ and $b_{1,k,\ell}$, determine the full Taylor series
of $U$, and hence the Taylor series of $V$ at the bottom of the well.
Moreover,
\begin{equation}
\label{eq:a30}
a_{30}^2=2v_1b_{1,0,0}
 +\frac{v_2^2}{v_1^2-4v_2^2}\,a_{12}^2.
\end{equation}
\end{theorem}

\begin{remark}
The auxiliary sequence $\{a_{1,2k}\}_{k\ge1}$ is equivalent to the Taylor
series of
\[
\partial_{x_1}U(0,x_2)=\sum_{k\ge1}a_{1,2k}x_2^{2k}.
\]
By \eqref{eq:linear_normalization}--\eqref{eq:normalized_U}, this amounts to
prescribing the corresponding transverse derivative data for the original
potential $V$. The result is therefore not a pure spectral rigidity theorem;
it is a Cauchy-type formal reconstruction result with prescribed
transverse-line data.
\end{remark}

As a direct consequence, we obtain the following inverse spectral statement.

\begin{corollary}
\label{cor:spectral}
Let $V$ satisfy the assumptions above, including the reflection symmetry
\eqref{eq:symmetry}, and label $v_1,v_2$ as specified after
\eqref{eq:normalized_U}. Suppose that $v_1/v_2\notin\Q$ and $a_{30}\ne0$.
Given $\operatorname{sgn}(a_{30})$ and $\{a_{1,2k}\}_{k\ge1}$, the family of
low-lying spectra of \eqref{eq:Schrodinger} in a sufficiently small fixed
interval $[0,\delta]$, for all sufficiently small $\hbar>0$ and including
multiplicities, uniquely determines the full Taylor series of $V$ at the origin.
\end{corollary}

\begin{proof}
In the nonresonant case, the bottom of the well spectral determination results
of \cite{GU07,GPU07,CdV09} imply that this family of exact spectra determines
the QBNF. Theorem~\ref{thm:main} then recovers the Taylor series of $U$, and
\eqref{eq:normalized_U} converts its coefficients into those of $V$.
\end{proof}

The class treated by Hezari \cite{Hez09} contains no mixed odd terms of the form
$x_1x_2^{2k}$ after the exceptional coordinate is labeled $x_1$; in the present
notation, this corresponds to $a_{1,2k}=0$. Here that sequence need not vanish,
but it is prescribed as transverse-line data. Compared with Guillemin--Uribe
\cite{GU11}, we retain only \eqref{eq:symmetry}: neither the axial condition
$V(q_1,0)=V(-q_1,0)$ nor the genericity condition corresponding to
$\partial_{q_1}\partial_{q_2}^2V(0)\ne0$ is assumed. The nondegeneracy condition
used here is instead $a_{30}\ne0$.

The rest of the paper is organized as follows. Section~\ref{sec:preliminaries}
reviews the Moyal product, the Weyl algebra, and the quantum Birkhoff normal
form. Section~\ref{sec:proof} proves Theorem~\ref{thm:main}. The projected
derivative identities used in the higher-order calculation are established in
Appendix~\ref{app:explicit_calc}.

\section{\label{sec:preliminaries}Preliminaries}

We work formally on $T^*\R^2$ and collect the notation used in the proof.

\subsection{Moyal product and Weyl algebra}

For symbols $A,B\in C^\infty(\R^{2d})$, their Moyal product is
\begin{equation}
\label{eq:moyal_product}
A\star B
=\left.
\exp\left[\frac{\mathrm i\hbar}{2}
\left(\partial_x\cdot\partial_\eta-\partial_\xi\cdot\partial_y\right)\right]
A(x,\xi)B(y,\eta)
\right|_{(y,\eta)=(x,\xi)}.
\end{equation}
For $j\ge0$, set
\[
\{A,B\}_j
=\left.
\left(\partial_x\cdot\partial_\eta-\partial_\xi\cdot\partial_y\right)^j
A(x,\xi)B(y,\eta)
\right|_{(y,\eta)=(x,\xi)}.
\]
Thus $\{A,B\}_0=AB$, $\{A,B\}_1=\{A,B\}$ is the classical Poisson bracket, and
\[
A\star B=\sum_{j=0}^\infty\frac1{j!}
\left(\frac{\mathrm i\hbar}{2}\right)^j\{A,B\}_j.
\]
The Moyal bracket $[A,B]^\star:=A\star B-B\star A$ contains only odd powers of
$\hbar$:
\begin{equation}
\label{eq:moyal_bracket}
\frac1{\mathrm i\hbar}[A,B]^\star
=\{A,B\}-\frac{\hbar^2}{24}\{A,B\}_3
 +\frac{\hbar^4}{1920}\{A,B\}_5-\cdots.
\end{equation}
When $W=W(x)$ depends only on the position variables, the third-order bracket in
dimension two is
\begin{equation}
\label{eq:position_bracket}
\begin{split}
\{A,W\}_3=-\bigl(&\partial_{\xi_1}^3A\,\partial_{x_1}^3W
 +3\partial_{\xi_1}^2\partial_{\xi_2}A\,
   \partial_{x_1}^2\partial_{x_2}W\\
&+3\partial_{\xi_1}\partial_{\xi_2}^2A\,
   \partial_{x_1}\partial_{x_2}^2W
 +\partial_{\xi_2}^3A\,\partial_{x_2}^3W\bigr).
\end{split}
\end{equation}

Let $\mathcal W_m$ be the space of weighted homogeneous polynomials of degree
$m$ in $x,\xi,\hbar$, where
\[
\deg x_j=\deg\xi_j=1,
\qquad
\deg\hbar=2.
\]
The formal Weyl algebra is the completed graded space
\[
\mathcal W=\prod_{m\ge0}\mathcal W_m,
\qquad
\mathcal W_m\star\mathcal W_n\subset\mathcal W_{m+n}.
\]
We write $F=\mathcal O_{\mathrm w}(M)$ when every weighted homogeneous component
of $F$ has degree at least $M$.

\subsection{Resonant terms}

Introduce complex coordinates
\[
z_j=x_j+\mathrm i\xi_j,
\qquad
\bar z_j=x_j-\mathrm i\xi_j,
\qquad j=1,2.
\]
For $\alpha,\beta\in\mathbb N_0^2$, write
$z^\alpha=z_1^{\alpha_1}z_2^{\alpha_2}$ and similarly for $\bar z^\beta$. Recall
that
\[
H_2=\frac12(v_1\Omega_1+v_2\Omega_2),
\qquad
\Omega_j=x_j^2+\xi_j^2=z_j\bar z_j.
\]
Define $L(P)=\{P,H_2\}$. Since $H_2$ is quadratic,
\[
\frac{\mathrm i}{\hbar}[P,H_2]^\star=-L(P),
\]
and a direct computation gives
\begin{equation}
\label{eq:L_action}
L(z^\alpha\bar z^\beta)
=-\mathrm i\langle v,\alpha-\beta\rangle z^\alpha\bar z^\beta,
\qquad
\langle v,\alpha-\beta\rangle
=v_1(\alpha_1-\beta_1)+v_2(\alpha_2-\beta_2).
\end{equation}
Because $v_1/v_2\notin\Q$, the frequency
$\langle v,\alpha-\beta\rangle$ vanishes if and only if $\alpha=\beta$. Since
$L$ does not act on $\hbar$,
\[
\ker L\cap\mathcal W_m
=\operatorname{span}\{\hbar^r z^\alpha\bar z^\alpha:
 r\ge0,\ 2r+2|\alpha|=m\},
\]
and hence
\[
\ker L=\mathbb C[[\hbar,\Omega_1,\Omega_2]].
\]

Define the resonant projection $\Pi$ coefficientwise by retaining exactly the
monomials $\hbar^rz^\alpha\bar z^\beta$ with $\alpha=\beta$. Every homogeneous
$F\in\mathcal W_m$ has a unique decomposition
\[
F=\Pi F+L(G),
\qquad
\Pi G=0.
\]
If, at a fixed power of $\hbar$,
$F-\Pi F=\sum_{\alpha\ne\beta}c_{\alpha\beta}z^\alpha\bar z^\beta$, then
\begin{equation}
\label{eq:L_inverse}
G=L^{-1}(F-\Pi F)
=\mathrm i\sum_{\alpha\ne\beta}
\frac{c_{\alpha\beta}}{\langle v,\alpha-\beta\rangle}
 z^\alpha\bar z^\beta.
\end{equation}
The formula is applied coefficientwise in $\hbar$.

\subsection{Quantum Birkhoff normal form and the homological equation}

Write
\[
H=H_2+U_3+U_4+\cdots,
\]
where $U_m$ is homogeneous of degree $m$ in $x$. The formal construction seeks
$S=S_3+S_4+\cdots$, with $S_m\in\mathcal W_m$ and $\Pi S_m=0$, such that
\begin{equation}
\label{eq:BCH}
e^{\frac{\mathrm i}{\hbar}\operatorname{ad}^\star_S}H
=H+\frac{\mathrm i}{\hbar}[S,H]^\star
 +\frac1{2!}\left(\frac{\mathrm i}{\hbar}\right)^2
 [S,[S,H]^\star]^\star+\cdots
\equiv H_2+\sum_{m\ge3}B_m.
\end{equation}
Matching terms of weighted degree $m$ gives
\[
B_m=F_m-L(S_m),
\]
where $F_m$ depends on $U_3,\dots,U_m$ and on the previously determined
generators $S_3,\dots,S_{m-1}$. Therefore
\begin{equation}
\label{eq:homological_equation}
B_m=\Pi F_m,
\qquad
S_m=L^{-1}(F_m-B_m).
\end{equation}
The normalization $\Pi S_m=0$ makes $S_m$ unique.

Every resonant monomial has even weighted degree, so $B_{2N-1}=0$. Moreover,
since the initial Hamiltonian is independent of $\hbar$ and
\eqref{eq:moyal_bracket} contains only even powers after division by
$\mathrm i\hbar$, the QBNF has the form
\[
B=H_2+\sum_{N=2}^\infty B_{2N}
=H_2+\sum_{2r+k+\ell\ge2}b_{r,k,\ell}\,
\hbar^{2r}\Omega_1^k\Omega_2^\ell.
\]
We write
\begin{equation}
\label{eq:hbar_layers}
S_m=S_m^0+\hbar^2S_m^2+\cdots,
\qquad
B_m=B_m^0+\hbar^2B_m^2+\cdots,
\end{equation}
where $S_m^{2r}$ and $B_m^{2r}$ are homogeneous polynomials in $(x,\xi)$ of
degree $m-4r$.

\section{\label{sec:proof}Proof of Theorem~\ref{thm:main}}

We begin with an elementary observation about even-degree terms.

\begin{lemma}
\label{lem:v2n}
Suppose that $U$ satisfies the reflection symmetry
$U(x_1,x_2)=U(x_1,-x_2)$. Then its
homogeneous component of degree $2N$,
\[
U_{2N} = \sum_{k=0}^{N} a_{2N-2k,2k}\, x_1^{2N-2k}x_2^{2k},
\]
is uniquely determined by its resonant projection $\Pi(U_{2N})$.
\end{lemma}
\begin{proof}
Using $x_j=(z_j+\bar z_j)/2$, we expand a
generic monomial by the binomial theorem:
\begin{align*}
x_1^{2N-2k}x_2^{2k}
&= \frac{1}{2^{2N}} (z_1+\bar{z}_1)^{2N-2k} (z_2+\bar{z}_2)^{2k} \\
&= \frac{1}{2^{2N}} \sum_{l=0}^{2N-2k} \sum_{j=0}^{2k}
   \binom{2N-2k}{l}\binom{2k}{j}
   z_1^{2N-2k-l}\bar{z}_1^{l}\, z_2^{2k-j}\bar{z}_2^{j}.
\end{align*}
The resonant terms satisfy $\alpha=\beta$; equivalently, the
powers of $z$ and $\bar{z}$ must match in each coordinate:
\[
2N-2k-l=l \Longleftrightarrow l=N-k,
\qquad
2k-j=j \Longleftrightarrow j=k.
\]
The coefficient of this resonant monomial is
\[
\frac{1}{2^{2N}} \binom{2N-2k}{N-k}\binom{2k}{k}
  z_1^{N-k}\bar{z}_1^{N-k} z_2^{k}\bar{z}_2^{k}
= \frac{1}{2^{2N}} \binom{2N-2k}{N-k}\binom{2k}{k} \Omega_1^{N-k}\Omega_2^{k}.
\]
Since
$\binom{2N-2k}{N-k}\binom{2k}{k}>0$, each coefficient
$a_{2N-2k,2k}$ can be recovered uniquely
from $\Pi(U_{2N})$.
\end{proof}

\subsection{\label{subsec:low}Base case of induction}

\proofstep{Step 1. Homological equations in degrees $3$ and $4$.}

Substitute $H=H_2+U_3+U_4+\cdots$ and
$S=S_3+S_4+\cdots$ into the BCH expansion and compare terms of weighted
degrees $3$ and $4$. Using the
vanishing of odd-degree resonant terms ($B_3=0$) and the identity
$\frac{\mathrm i}{\hbar}[S_3,H_2]^\star=-L(S_3)$, we obtain
the homological equations
\begin{align}
  U_3 - L(S_3) &= 0, \label{eq:m3} \\
  U_4 + \frac{\mathrm{i}}{\hbar}[S_4,H_2]^\star
      + \frac{\mathrm{i}}{\hbar}[S_3,U_3]^\star
      + \frac{1}{2}\Bigl(\frac{\mathrm{i}}{\hbar}\Bigr)^2
        [S_3,[S_3,H_2]^\star]^\star
      &= B_4. \label{eq:m4_raw}
\end{align}

\proofstep{Step 2. Simplification of the degree-$4$ equation.}

From \eqref{eq:m3},
$[S_3,H_2]^\star=\mathrm i\hbar U_3$, and therefore
\[
\frac12\left(\frac{\mathrm i}{\hbar}\right)^2
 [S_3,[S_3,H_2]^\star]^\star
=-\frac12\frac{\mathrm i}{\hbar}[S_3,U_3]^\star.
\]
Because $S_3$ and $U_3$ are cubic, their fifth and higher Moyal brackets
vanish. Thus
\[
\frac{\mathrm{i}}{\hbar}[S_3,U_3]^\star
  = -\{S_3,U_3\} + \frac{\hbar^2}{24}\{S_3,U_3\}_3.
\]
Substitution into \eqref{eq:m4_raw} gives
\begin{equation}
\label{eq:m4}
  U_4 - L(S_4) - \frac12\{S_3,U_3\} + \frac{\hbar^2}{48}\{S_3,U_3\}_3
  = A\Omega_1^2 + B\Omega_1\Omega_2 + C\Omega_2^2 + D\hbar^2,
\end{equation}
where $D = b_{1,0,0}$.

\proofstep{Step 3. Reconstruction of $U_3$ and $S_3$.}

Taking the $\hbar^2$ part of \eqref{eq:m4} and applying $\Pi$ gives
\begin{equation}
\label{eq:b100_projection}
 b_{1,0,0}=\frac1{48}\Pi\{S_3,U_3\}_3.
\end{equation}
We first solve the degree-$3$ homological equation. By reflection symmetry,
\[
U_3=a_{30}x_1^3+a_{12}x_1x_2^2,
\qquad
S_3=L^{-1}U_3.
\]

For the $x_1^3$ term,
\[
L(\xi_1^3)=-3v_1x_1\xi_1^2,
\qquad
L(x_1^2\xi_1)=v_1(2x_1\xi_1^2-x_1^3),
\]
so
$L(x_1^2\xi_1+\frac23\xi_1^3)=-v_1x_1^3$.

For the $x_1x_2^2$ term, seek a preimage of the form
$A\xi_1x_2^2+B\xi_1\xi_2^2+Fx_1x_2\xi_2$. Applying $L$ and comparing
coefficients gives
\begin{equation*}
  -v_1 A - v_2 F = 1, \quad
  -v_1 B + v_2 F = 0, \quad
  2v_2 A - 2v_2 B + v_1 F = 0.
\end{equation*}
Solving gives
\[
B=\frac{2v_2^2}{v_1(v_1^2-4v_2^2)},
\qquad
F=\frac{v_1}{v_2}B,
\qquad
A=-\frac1{v_1}-B.
\]
The denominator is nonzero because $v_1/v_2\notin\Q$.

Thus
\begin{equation}
\label{eq:S3_explicit}
\begin{aligned}
S_3={}&-\frac{a_{30}}{v_1}
 \left(x_1^2\xi_1+\frac23\xi_1^3\right)\\
&+a_{12}\left(
 \frac{2v_2^2-v_1^2}{v_1(v_1^2-4v_2^2)}\xi_1x_2^2
 +\frac{2v_2^2}{v_1(v_1^2-4v_2^2)}\xi_1\xi_2^2
 +\frac{2v_2}{v_1^2-4v_2^2}x_1x_2\xi_2
\right).
\end{aligned}
\end{equation}

Since $U_3$ depends only on $x$, \eqref{eq:position_bracket} reduces to
\begin{equation*}
  \{S_3, U_3\}_3 = - \left(
    \partial_{\xi_1}^3 S_3\,\partial_{x_1}^3 U_3
    + 3\partial_{\xi_1}\partial_{\xi_2}^2 S_3\,\partial_{x_1}\partial_{x_2}^2 U_3
    \right).
\end{equation*}
Inserting the derivatives
$\partial_{x_1}^3 U_3 = 6a_{30}$,
$\partial_{\xi_1}^3 S_3 = -\frac{4a_{30}}{v_1}$,
$\partial_{x_1}\partial_{x_2}^2 U_3 = 2a_{12}$,
$\partial_{\xi_1}\partial_{\xi_2}^2 S_3
 = \frac{4v_2^2 a_{12}}{v_1(v_1^2 - 4v_2^2)}$,
we compute
\begin{align*}
  \{S_3, U_3\}_3
  &= - \left[ \left(-\frac{4a_{30}}{v_1}\right)(6a_{30})
       + 3 \left(\frac{4v_2^2 a_{12}}{v_1(v_1^2 - 4v_2^2)}\right)(2a_{12}) \right] \\
  &= \frac{24}{v_1} \left( a_{30}^2 - \frac{v_2^2}{v_1^2 - 4v_2^2} a_{12}^2 \right).
\end{align*}
This expression is constant, so it is already resonant. Substituting it into
\eqref{eq:b100_projection} gives
\begin{equation*}
  D = b_{1,0,0} = \frac{1}{48} \{S_3, U_3\}_3
    = \frac{1}{2v_1} \left( a_{30}^2 - \frac{v_2^2}{v_1^2 - 4v_2^2} a_{12}^2 \right).
\end{equation*}
Equivalently,
\begin{equation*}
  a_{30}^2 = 2v_1 b_{1,0,0} + \frac{v_2^2}{v_1^2 - 4v_2^2} a_{12}^2.
\end{equation*}
The prescribed sign of $a_{30}$ therefore determines $a_{30}$ itself. Hence
$U_3$ and $S_3$ are determined.

\proofstep{Step 4. Reconstruction of $U_4$ from its resonant projection.}

The $\hbar^0$ part of \eqref{eq:m4} is
\[
  U_4 - L(S_4^0) -\frac12\{S_3,U_3\}
  = B_4^0 = A\Omega_1^2 + B\Omega_1\Omega_2 + C\Omega_2^2.
\]
Applying the resonant projection $\Pi$ to both sides and using
$\Pi(L(S_4^0))=0$, we obtain
\[
  \Pi(U_4) = B_4^0 + \frac12\Pi(\{S_3,U_3\}).
\]
The right-hand side is known. By reflection symmetry,
$U_4=a_{40}x_1^4+a_{22}x_1^2x_2^2+a_{04}x_2^4$, and
\[
  \Pi(U_4) = \frac{3}{8}a_{40}\Omega_1^2 + \frac{1}{4}a_{22}\Omega_1\Omega_2
           + \frac{3}{8}a_{04}\Omega_2^2.
\]
Thus $a_{40}$, $a_{22}$, and $a_{04}$ are uniquely determined, in agreement
with Lemma~\ref{lem:v2n}.

\proofstep{Step 5. Reconstruction of the first two $\hbar$-layers of $S_4$.}

The $\hbar^0$ and $\hbar^2$ parts of \eqref{eq:m4} now have known
right-hand sides. Taking their nonresonant parts and imposing
$\Pi S_4^0=\Pi S_4^2=0$ determines $S_4^0$ and $S_4^2$ uniquely. This
establishes the base case.

\subsection{\label{subsec:induction}Induction step}

\proofstep{Step 1. Homological equations in degrees $2N-1$ and $2N$.}

Let $N\ge3$. Suppose that the Taylor expansion of $U$ is already known up to
order $2N-2$, together with the $\hbar^0$ and $\hbar^2$ layers of
$S_3,\dots,S_{2N-2}$. Put $S'=S_3+\cdots+S_{2N-2}$. Then
\begin{equation}
\label{eq:truncated_conjugation}
 e^{\frac{\mathrm i}{\hbar}\mathrm{ad}^\star_{S'}}
 (H_2+U_3+\cdots+U_{2N-2})
 =H_2+B_4+\cdots+B_{2N-2}+R_{2N-1}+R_{2N}
 +\mathcal O_{\mathrm w}(2N+1),
\end{equation}
where $R_{2N-1}$ and $R_{2N}$ depend only on the previously determined
data. In particular, their $\hbar^0$ and $\hbar^2$ layers depend only on the
corresponding layers of $S_3,\dots,S_{2N-2}$.

For weighted homogeneous symbols $A\in\mathcal W_a$ and
$C\in\mathcal W_c$, set
$D_A C=\frac{\mathrm i}{\hbar}[A,C]^\star$. Then $D_A C$ has weighted
degree $a+c-2$, and an iterated bracket
$D_{A_1}\cdots D_{A_r}C$, with $A_j\in\mathcal W_{a_j}$, has degree
\[
c+\sum_{j=1}^r(a_j-2).
\]
Consequently, through degree $2N$, the only terms involving the new symbols
$U_{2N-1}$, $U_{2N}$, $S_{2N-1}$, and $S_{2N}$ are those displayed below;
all other terms depend only on the induction data and are absorbed into the
remainders.

Let $S=S'+S_{2N-1}+S_{2N}$ normalize the Hamiltonian through degree $2N$.
Comparing terms of degrees $2N-1$ and $2N$ gives
\begin{equation}
\label{eq:odd}
U_{2N-1}+\frac{\mathrm i}{\hbar}[S_{2N-1},H_2]^\star
+R_{2N-1}=0
\end{equation}
and
\begin{align}
\label{eq:even}
\begin{split}
U_{2N}-L(S_{2N})
&+\frac{\mathrm i}{\hbar}[S_{2N-1},U_3]^\star
 +\frac{\mathrm i}{\hbar}[S_3,U_{2N-1}]^\star\\
&+\frac12\left(\frac{\mathrm i}{\hbar}\right)^2
 \Bigl([S_{2N-1},[S_3,H_2]^\star]^\star
       +[S_3,[S_{2N-1},H_2]^\star]^\star\Bigr)\\
&+R_{2N}=B_{2N}.
\end{split}
\end{align}
Here and below, the remainder notation refers only to terms already known from
the induction hypothesis.

\proofstep{Determination of $S_{2N-1}$.}
Write $S_{2N-1}=S_{2N-1}^0+\hbar^2S_{2N-1}^2+\cdots$. The first two layers of
\eqref{eq:odd} are
\begin{align}
U_{2N-1}-L(S_{2N-1}^0)+R_{2N-1}^0&=0,
\label{eq:odd_h0}\\
-L(S_{2N-1}^2)+R_{2N-1}^2&=0.
\label{eq:odd_h2}
\end{align}
Since the total degree is odd, there is no resonant obstruction. Therefore
\begin{equation}
\label{eq:Sodd_split}
S_{2N-1}^0=L^{-1}U_{2N-1}+L^{-1}R_{2N-1}^0,
\qquad
S_{2N-1}^2=L^{-1}R_{2N-1}^2.
\end{equation}
The second identity and the second summand in the first identity are known.

By the reflection symmetry,
\begin{equation}
\label{eq:U_odd}
U_{2N-1}=\sum_{m=0}^{N-1}a_{2N-1-2m,2m}
 x_1^{2N-1-2m}x_2^{2m}.
\end{equation}
In complex coordinates this is
\begin{align}
U_{2N-1}
={}&\sum_{m=0}^{N-1}\frac{a_{2N-1-2m,2m}}{2^{2N-1}}
\sum_{l=0}^{2N-1-2m}\sum_{j=0}^{2m}
\binom{2N-1-2m}{l}\binom{2m}{j}\notag\\
&\qquad\times
z^{(2N-1-2m-l,2m-j)}\bar z^{(l,j)}.
\label{eq:U_odd_complex}
\end{align}
Set
\[
\lambda_{m,l,j}=v_1(2N-1-2m-2l)+v_2(2m-2j).
\]
Using \eqref{eq:L_inverse}, the unknown part of $S_{2N-1}^0$ is
\begin{align}
L^{-1}U_{2N-1}
={}&\mathrm i\sum_{m=0}^{N-1}
\frac{a_{2N-1-2m,2m}}{2^{2N-1}}
\sum_{l=0}^{2N-1-2m}\sum_{j=0}^{2m}
\frac{\binom{2N-1-2m}{l}\binom{2m}{j}}{\lambda_{m,l,j}}
\notag\\
&\qquad\times
z^{(2N-1-2m-l,2m-j)}\bar z^{(l,j)}.
\label{eq:S0}
\end{align}
No denominator in \eqref{eq:S0} vanishes. Otherwise the corresponding monomial
would be resonant, which is impossible in odd weighted degree.

\proofstep{Step 2. Simplification of the degree-$2N$ equation.}

From the equations of degrees $3$ and $2N-1$,
\[
[S_3,H_2]^\star=\mathrm i\hbar U_3,
\qquad
[S_{2N-1},H_2]^\star
=\mathrm i\hbar(U_{2N-1}+R_{2N-1}).
\]
Substituting these identities into \eqref{eq:even} and absorbing all known
contributions into a new remainder $\widetilde R_{2N}$ gives
\begin{equation}
\label{eq:even_merged}
U_{2N}-L(S_{2N})
+\frac12\frac{\mathrm i}{\hbar}[S_3,U_{2N-1}]^\star
+\frac12\frac{\mathrm i}{\hbar}[S_{2N-1},U_3]^\star
+\widetilde R_{2N}=B_{2N}.
\end{equation}
The $\hbar^0$ and $\hbar^2$ layers of $\widetilde R_{2N}$ are known. Taking the
$\hbar^2$ layer of \eqref{eq:even_merged}, applying $\Pi$, and using
\eqref{eq:Sodd_split}, we obtain
\begin{equation}
\label{eq:h2_reduced}
B_{2N}^2=\frac1{48}\Pi\Bigl(
\{S_3,U_{2N-1}\}_3+\{L^{-1}U_{2N-1},U_3\}_3\Bigr)
+K_{2N}^2,
\end{equation}
where $K_{2N}^2$ denotes the sum of the projected $\hbar^2$ terms depending
only on the induction data.

\proofstep{Step 3. Reconstruction of $U_{2N-1}$ and $S_{2N-1}$.}

For a resonant polynomial $F$, let $[F]_{r,s}$ denote the coefficient of
$\Omega_1^r\Omega_2^s$. The next lemma exhibits the triangular structure of
the unknown part of \eqref{eq:h2_reduced}.

\begin{lemma}
\label{lem:triangular}
Let $0\le m\le N-2$ and put $n=N-m$. Set
\[
A_m=a_{2N-1-2m,2m}=a_{2n-1,2m},
\qquad A_{m+1}=a_{2n-3,2m+2}.
\]
Then
\begin{equation}
\label{eq:triangular_relation}
\left[\Pi\Bigl(
\{S_3,U_{2N-1}\}_3+\{L^{-1}U_{2N-1},U_3\}_3
\Bigr)\right]_{n-2,m}
=C_{N,m}A_m+G_{N,m}A_{m+1},
\end{equation}
where
\begin{equation}
\label{eq:C_Nm}
C_{N,m}=\frac{64a_{30}}{v_1\,2^{2N-1}}
\binom{2m}{m}\binom{2n-1}{n}n(n-1)^2
\end{equation}
and
\begin{equation}
\label{eq:G_Nm}
\begin{split}
G_{N,m}={}&-\frac{24a_{12}v_2^2}
{v_1(v_1^2-4v_2^2)\,2^{2N-4}}
(2n-3)(2m+2)(2m+1)\\
&\times\binom{2n-4}{n-2}\binom{2m}{m}.
\end{split}
\end{equation}
No other coefficient of $U_{2N-1}$ occurs in
\eqref{eq:triangular_relation}.
\end{lemma}

\begin{proof}
For $p,q\ge0$, let $\mathcal P_{p,q}$ denote the polynomials homogeneous of
degree $p$ in $(x_1,\xi_1)$ and degree $q$ in $(x_2,\xi_2)$. The formula
\[
L=v_1(\xi_1\partial_{x_1}-x_1\partial_{\xi_1})
 +v_2(\xi_2\partial_{x_2}-x_2\partial_{\xi_2})
\]
shows that $L$ preserves each $\mathcal P_{p,q}$. The same is true of $\Pi$,
and $L^{-1}$ preserves the bidegree on the nonresonant subspace.

From \eqref{eq:S3_explicit} and \eqref{eq:position_bracket},
\begin{equation}
\label{eq:first_unknown_bracket}
\{S_3,U_{2N-1}\}_3
=\frac{4a_{30}}{v_1}\partial_{x_1}^3U_{2N-1}
-\frac{12a_{12}v_2^2}{v_1(v_1^2-4v_2^2)}
 \partial_{x_1}\partial_{x_2}^2U_{2N-1},
\end{equation}
while
\begin{equation}
\label{eq:second_unknown_bracket}
\{L^{-1}U_{2N-1},U_3\}_3
=-6a_{30}\partial_{\xi_1}^3L^{-1}U_{2N-1}
-6a_{12}\partial_{\xi_1}\partial_{\xi_2}^2L^{-1}U_{2N-1}.
\end{equation}

The target monomial $\Omega_1^{n-2}\Omega_2^m$ has bidegree
$(2n-4,2m)$. In both \eqref{eq:first_unknown_bracket} and
\eqref{eq:second_unknown_bracket}, the pure third derivative in the first
coordinate lowers the bidegree by $(3,0)$. Hence it can contribute to the
target coefficient only from
\[
A_mx_1^{2n-1}x_2^{2m}
\quad\text{or}\quad
L^{-1}(A_mx_1^{2n-1}x_2^{2m}).
\]
Similarly, the mixed derivative lowers the bidegree by $(1,2)$, so it can
contribute only from
\[
A_{m+1}x_1^{2n-3}x_2^{2m+2}
\quad\text{or}\quad
L^{-1}(A_{m+1}x_1^{2n-3}x_2^{2m+2}).
\]
Therefore no coefficient other than $A_m$ and $A_{m+1}$ can occur in
\eqref{eq:triangular_relation}.

Set
\[
\kappa_{n,m}
=\frac1{2^{2N-4}}\binom{2n-4}{n-2}\binom{2m}{m},
\qquad
\Pi(x_1^{2n-4}x_2^{2m})
=\kappa_{n,m}\Omega_1^{n-2}\Omega_2^m.
\]
We now compute separately the coefficients multiplying $A_m$ and
$A_{m+1}$.

For $A_m$, the first term of \eqref{eq:first_unknown_bracket} gives
\[
\frac{4a_{30}}{v_1}
(2n-1)(2n-2)(2n-3)\kappa_{n,m}A_m.
\]
Identity \eqref{eq:projected_diag} gives
\begin{align*}
&-6a_{30}
\left[\Pi\partial_{\xi_1}^3L^{-1}
\bigl(A_mx_1^{2n-1}x_2^{2m}\bigr)\right]_{n-2,m}\\
&\qquad=\frac{4a_{30}}{v_1}
(2n-1)(2n-2)(2n-3)\kappa_{n,m}A_m.
\end{align*}
Thus the two brackets make equal contributions to the coefficient of $A_m$.
Adding these two contributions, we obtain
\[
\frac{8a_{30}}{v_1}
(2n-1)(2n-2)(2n-3)\kappa_{n,m}A_m.
\]
The elementary identity
\[
(2n-1)(2n-2)(2n-3)\binom{2n-4}{n-2}
=n(n-1)^2\binom{2n-1}{n}
\]
converts the coefficient in the preceding display exactly into
$C_{N,m}$ in \eqref{eq:C_Nm}.

For $A_{m+1}$, the mixed term of
\eqref{eq:first_unknown_bracket} contributes
\[
-\frac{12a_{12}v_2^2}{v_1(v_1^2-4v_2^2)}
(2n-3)(2m+2)(2m+1)\kappa_{n,m}A_{m+1}.
\]
Using identity \eqref{eq:projected_cross}, the mixed term of
\eqref{eq:second_unknown_bracket} contributes the same quantity:
\begin{align*}
&-6a_{12}
\left[\Pi\partial_{\xi_1}\partial_{\xi_2}^2L^{-1}
\bigl(A_{m+1}x_1^{2n-3}x_2^{2m+2}\bigr)\right]_{n-2,m}\\
&\qquad=-\frac{12a_{12}v_2^2}{v_1(v_1^2-4v_2^2)}
(2n-3)(2m+2)(2m+1)\kappa_{n,m}A_{m+1}.
\end{align*}
Adding these two contributions and substituting the value of
$\kappa_{n,m}$ from its definition above gives precisely $G_{N,m}A_{m+1}$
with $G_{N,m}$ as in \eqref{eq:G_Nm}. Combining the $A_m$ and $A_{m+1}$
terms proves \eqref{eq:triangular_relation}.
\end{proof}

Returning to \eqref{eq:h2_reduced}, Lemma~\ref{lem:triangular} gives
\begin{equation}
\label{eq:recurrence}
C_{N,m}A_m+G_{N,m}A_{m+1}=F_{N,m},
\qquad 0\le m\le N-2,
\end{equation}
where
\[
F_{N,m}=48[B_{2N}^2-K_{2N}^2]_{n-2,m}
\]
is known. The endpoint $A_{N-1}=a_{1,2N-2}$ belongs to the prescribed
transverse-line data. Since $n\ge2$ and $a_{30}\ne0$, the coefficient
$C_{N,m}$ is nonzero. Starting with $m=N-2$ and proceeding backward to $m=0$,
\eqref{eq:recurrence} uniquely determines all coefficients of
$U_{2N-1}$. Equations \eqref{eq:Sodd_split} then determine the first two layers
of $S_{2N-1}$.

\proofstep{Step 4. Reconstruction of $U_{2N}$ from its resonant projection.}
The $\hbar^0$ part of \eqref{eq:even_merged} is
\[
U_{2N}-L(S_{2N}^0)-\frac12\{S_3,U_{2N-1}\}
-\frac12\{S_{2N-1}^0,U_3\}+\widetilde R_{2N}^0=B_{2N}^0.
\]
Applying the resonant projection gives
\begin{equation}
\label{eq:even_reconstruction}
\Pi(U_{2N})=B_{2N}^0
+\frac12\Pi\bigl(\{S_3,U_{2N-1}\}+\{S_{2N-1}^0,U_3\}\bigr)
-\Pi(\widetilde R_{2N}^0).
\end{equation}
The right-hand side is known. By Lemma~\ref{lem:v2n}, $U_{2N}$ is uniquely
determined.

\proofstep{Step 5. Reconstruction of the first two $\hbar$-layers of $S_{2N}$.}

The nonresonant parts of the two homological equations, together with
$\Pi S_{2N}^0=\Pi S_{2N}^2=0$, determine the $\hbar^0$ and $\hbar^2$ layers of
$S_{2N}$. This completes the induction and proves
Theorem~\ref{thm:main}.

\section{Conclusion}

For a two-dimensional semiclassical Schr\"odinger operator with a single
reflection symmetry, we have shown that the first two layers of the QBNF,
together with $\operatorname{sgn}(a_{30})$ and the transverse coefficients
$\{a_{1,2k}\}_{k\ge1}$, determine the complete Taylor series of the potential
at the bottom of the well. The reconstruction is degree by degree: the
$\hbar^2$ layer yields a triangular system for the odd homogeneous terms, and
the classical layer then recovers the even homogeneous terms from their
resonant projections. Thus the transverse sequence plays the role of Cauchy
data for a constructive two-layer inverse problem.

\appendix
\section{\label{app:explicit_calc}Projected derivative identities}

The following identities provide the coefficients used in
Lemma~\ref{lem:triangular}.

\begin{lemma}
\label{lem:projected_derivatives}
Let $n\ge2$, $m\ge0$, and $N=n+m$. Then
\begin{equation}
\label{eq:projected_diag}
\begin{split}
\Pi\partial_{\xi_1}^3L^{-1}
\bigl(x_1^{2n-1}x_2^{2m}\bigr)
={}&-\frac{2}{3v_1}(2n-1)(2n-2)(2n-3)\\
&\times\Pi\bigl(x_1^{2n-4}x_2^{2m}\bigr),
\end{split}
\end{equation}
and
\begin{equation}
\label{eq:projected_cross}
\begin{split}
\Pi\partial_{\xi_1}\partial_{\xi_2}^2L^{-1}
\bigl(x_1^{2n-3}x_2^{2m+2}\bigr)
={}&\frac{2v_2^2}{v_1(v_1^2-4v_2^2)}
(2n-3)(2m+2)(2m+1)\\
&\times\Pi\bigl(x_1^{2n-4}x_2^{2m}\bigr).
\end{split}
\end{equation}
\end{lemma}

\begin{proof}
We use
\[
x_j=\frac{z_j+\bar z_j}{2},
\qquad
\partial_{\xi_j}=\mathrm i(\partial_{z_j}-\partial_{\bar z_j}),
\]
together with \eqref{eq:L_inverse}. For a nonnegative integer $a$, let
$(a)_r=a(a-1)\cdots(a-r+1)$, with $(a)_0=1$ and $(a)_r=0$ when $r>a$.
Notice that the factor $\mathrm i$ in $L^{-1}$ cancels the factor
$\mathrm i^3=-\mathrm i$ contributed by either third-order differential
operator below.

We first prove \eqref{eq:projected_diag}. Expand
\begin{equation}
\label{eq:diag_expansion}
x_1^{2n-1}x_2^{2m}
=\frac1{2^{2N-1}}
\sum_{l=0}^{2n-1}\sum_{j=0}^{2m}
\binom{2n-1}{l}\binom{2m}{j}
 z_1^{2n-1-l}\bar z_1^lz_2^{2m-j}\bar z_2^j.
\end{equation}
For the term indexed by $(l,j)$, set
\[
\lambda_{l,j}=v_1(2n-1-2l)+v_2(2m-2j).
\]
The total degree is odd, so $\lambda_{l,j}\ne0$. With
$D_1=\partial_{z_1}-\partial_{\bar z_1}$, one has
\[
D_1^3\bigl(z_1^{2n-1-l}\bar z_1^l\bigr)
=\sum_{r=0}^3(-1)^r\binom3r
(2n-1-l)_{3-r}(l)_r
z_1^{2n-4-l+r}\bar z_1^{l-r}.
\]
After differentiation, resonance in the second coordinate forces $j=m$, while
resonance in the first coordinate forces $l=n-2+r$. Thus exactly four terms
survive. Define
\[
D_{n,m}=\frac{(2n-1)!}{((n-2)!)^2}\binom{2m}{m}.
\]
After the common factor
$2^{-(2N-1)}\Omega_1^{n-2}\Omega_2^m$ is removed, the four scalar
contributions are
\begin{center}
\renewcommand{\arraystretch}{1.3}
\setlength{\tabcolsep}{1.5em}
\begin{tabular}{ccccc}
\toprule
$r$ & $0$ & $1$ & $2$ & $3$ \\
\midrule
Scalar contribution
& $\dfrac{1}{3v_1}D_{n,m}$
& $-\dfrac{3}{v_1}D_{n,m}$
& $-\dfrac{3}{v_1}D_{n,m}$
& $\dfrac{1}{3v_1}D_{n,m}$ \\
\bottomrule
\end{tabular}
\end{center}
Their sum gives
\begin{equation}
\label{eq:diag_projected_value}
\Pi\partial_{\xi_1}^3L^{-1}
\bigl(x_1^{2n-1}x_2^{2m}\bigr)
=-\frac{16}{3v_1\,2^{2N-1}}
\frac{(2n-1)!}{((n-2)!)^2}\binom{2m}{m}
\Omega_1^{n-2}\Omega_2^m.
\end{equation}
On the other hand,
\begin{equation}
\label{eq:common_projection}
\Pi\bigl(x_1^{2n-4}x_2^{2m}\bigr)
=\frac1{2^{2N-4}}
\binom{2n-4}{n-2}\binom{2m}{m}
\Omega_1^{n-2}\Omega_2^m.
\end{equation}
The ratio of the scalar coefficients in
\eqref{eq:diag_projected_value} and \eqref{eq:common_projection} is
\[
-\frac{2}{3v_1}(2n-1)(2n-2)(2n-3),
\]
which proves \eqref{eq:projected_diag}.

We next prove \eqref{eq:projected_cross}. Expand
\begin{equation}
\label{eq:cross_expansion}
x_1^{2n-3}x_2^{2m+2}
=\frac1{2^{2N-1}}
\sum_{l=0}^{2n-3}\sum_{j=0}^{2m+2}
\binom{2n-3}{l}\binom{2m+2}{j}
 z_1^{2n-3-l}\bar z_1^lz_2^{2m+2-j}\bar z_2^j.
\end{equation}
For the term indexed by $(l,j)$, set
\[
\mu_{l,j}=v_1(2n-3-2l)+v_2(2m+2-2j).
\]
Again $\mu_{l,j}\ne0$ because the total degree is odd. Let
$D_2=\partial_{z_2}-\partial_{\bar z_2}$. If $r\in\{0,1\}$ and
$s\in\{0,1,2\}$ count the derivatives falling on $\bar z_1$ and $\bar z_2$,
respectively, then the corresponding coefficient in $D_1D_2^2$ is
$(-1)^{r+s}\binom1r\binom2s$. Resonance forces
\[
l=n-2+r,
\qquad
j=m+s.
\]
Hence six terms survive. Define
\[
M_{n,m}=\frac{(2n-3)!}{((n-2)!)^2}
\frac{(2m+2)!}{(m!)^2}.
\]
After the common factor
$2^{-(2N-1)}\Omega_1^{n-2}\Omega_2^m$ is removed, their scalar
contributions are
\begin{center}
\renewcommand{\arraystretch}{1.3}
\setlength{\tabcolsep}{1.35em}
\begin{tabular}{cccc}
\toprule
$r\backslash s$ & $0$ & $1$ & $2$ \\
\midrule
$0$
& $\dfrac{1}{v_1+2v_2}M_{n,m}$
& $-\dfrac{2}{v_1}M_{n,m}$
& $\dfrac{1}{v_1-2v_2}M_{n,m}$ \\
$1$
& $\dfrac{1}{v_1-2v_2}M_{n,m}$
& $-\dfrac{2}{v_1}M_{n,m}$
& $\dfrac{1}{v_1+2v_2}M_{n,m}$ \\
\bottomrule
\end{tabular}
\end{center}
Their sum is
\[
2\left(\frac1{v_1-2v_2}-\frac2{v_1}+\frac1{v_1+2v_2}\right)M_{n,m}
=\frac{16v_2^2}{v_1(v_1^2-4v_2^2)}M_{n,m}.
\]
Therefore
\begin{equation}
\label{eq:cross_projected_value}
\begin{split}
\Pi\partial_{\xi_1}\partial_{\xi_2}^2L^{-1}
\bigl(x_1^{2n-3}x_2^{2m+2}\bigr)
={}&\frac{16v_2^2}{v_1(v_1^2-4v_2^2)\,2^{2N-1}}\\
&\times\frac{(2n-3)!}{((n-2)!)^2}
\frac{(2m+2)!}{(m!)^2}
\Omega_1^{n-2}\Omega_2^m.
\end{split}
\end{equation}
Dividing the scalar coefficient in \eqref{eq:cross_projected_value} by that in
\eqref{eq:common_projection} gives
\[
\frac{2v_2^2}{v_1(v_1^2-4v_2^2)}
(2n-3)(2m+2)(2m+1),
\]
which proves \eqref{eq:projected_cross}.
\end{proof}

\end{document}